\documentclass{amsart}
\usepackage[english]{babel}
\usepackage{amsmath}
\usepackage{amsfonts}
\usepackage{amsthm}
\usepackage{inputenc}
\usepackage{amssymb}
\usepackage{graphicx}

\newtheorem{thm}{Theorem}[section]

\newtheorem{conj}[thm]{Conjecture}
\newtheorem{prop}[thm]{Proposition}
\newtheorem{defn}[thm]{Definition}
\newtheorem{exam}[thm]{Example}

\numberwithin{equation}{section}

\begin{document}

\title[Local derivations on Solvable Lie algebras]{Local derivations on Solvable Lie algebras}%

\author{Ayupov Sh.A., Khudoyberdiyev A.Kh.}

\address{[Ayupov  \ Sh.\ A.]
Institute of Mathematics Academy of Sciences of Uzbekistan, 81, Mirzo Ulugbek str., Tashkent, 100170, Uzbekistan.}
\email{sh\_ayupov@mail.ru}

\address{[Khudoyberdiyev \ A.\ Kh.]
National University of Uzbekistan, Institute of Mathematics Academy of Sciences of Uzbekistan, Tashkent, 100174, Uzbekistan.}
\email{khabror@mail.ru}

\maketitle
\begin{abstract}
We show that in the class of solvable Lie algebras there exist algebras which admit local derivations which are not ordinary derivation
and also algebras for which every local derivation is a derivation.
We found necessary and sufficient conditions under which any local derivation of solvable Lie algebras with abelian nilradical and one-dimensional complementary space is a derivation. Moreover, we prove that every local derivation on a finite-dimensional solvable Lie algebra with model nilradical and
maximal dimension of complementary space is a derivation.

\end{abstract} \maketitle




\section{Introduction}

The notions of local derivations were first introduced in 1990 by R.V.
Kadison \cite{Kadison} and D.R.Larson, A.R.Sourour \cite{Larson}.
The main problems concerning this notion are to find conditions under which local derivations become derivations and to present examples of algebras with local derivations that are not derivations.
R.V.Kadison proves that each continuous local
derivation of a von Neumann algebra $M$ into a dual Banach $M$-bimodule is a derivation. In 2001 B.E.Johnson
culminated the studies on local derivations, showing that every
local derivation from a $C^*$-algebra $A$ into a Banach $A$-bimodule is a
derivation \cite{John}.

Investigation of local derivations on algebras of measurable operators were initiated in papers \cite{Ayupov2}, \cite{Ayupov3}, \cite{Bresar} and others. In particular, the paper \cite{Ayupov2} is devoted to the study of local derivations on the algebra $S(M, \tau)$ of $\tau$-measurable operators affiliated with a von Neumann algebra $M$ and a faithful normal semi-finite trace $\tau.$ It is proved that every local derivation on $S(M, \tau)$ which is continuous in the measure topology automatically becomes a derivation. The paper \cite{Ayupov2} also deals with the problem of existence of local derivations which are not derivations on algebras of measurable operators. Namely, necessary and sufficient conditions were obtained for the algebras of measurable and $\tau$-measurable operators affiliated with a commutative von Neumann algebra to admit local derivations that are not derivations.

Later, several papers have been devoted to similar notions and corresponding problems
for derivations and automorphisms of Lie algebras \cite{Ayupov1}, \cite{Chen}.
In \cite{Ayupov1} Sh.A.Ayupov and K.K. Kudaybergenov
have proved that every local
derivation on semi-simple Lie algebras is a derivation and gave examples of nilpotent
finite-dimensional Lie algebras with local derivations which are not derivations.
The paper \cite{Ayupov4} devoted to investigation of so-called 2-local derivations on finite-dimensional Lie algebras and it is proved
that every 2-local derivation on a semi-simple Lie algebra $L$ is a derivation and that
each finite-dimensional nilpotent Lie algebra with dimension larger than two admits
2-local derivation which is not a derivation.

It is well known that any finite-dimensional Lie algebra over a field of characteristic zero is decomposed into the semidirect sum of semi-simple subalgebra and solvable radical.
The semi-simple part is a direct sum of simple Lie algebras which are completely classified, and
solvable Lie algebras can be classified by means of its nilradical.
There are several papers which deal with
the problem of classification of all solvable Lie algebras with a given nilradical, for example
Abelian, Heisenberg, filiform, quasi-filiform nilradicals, etc. \cite{AnCaGa}, \cite{NdWi}, \cite{RubWi}.
It should be noted that any derivation on a semi-simple Lie algebra is inner and any nilpotent Lie algebra has a derivation which is not inner. In the class of solvable Lie algebras there exist algebras which any derivation is inner and also algebras which admits not inner derivation.

In this paper we investigate local derivations of solvable Lie algebras.
We show that in the class of solvable Lie algebras there exists a solvable algebra admitting local derivations which are not ordinary derivation
and also there exist algebras for which every local derivation is a derivation.

More precisely, local derivations of solvable Lie algebras with abelian nilradical and one-dimensional complementary space are investigated.
A necessary and sufficient conditions under which exery local derivation of such Lie algebras becomes a derivation are found. We also consider solvable Lie algebra with model nilradical and maximal dimension of complementary space and prove that any local derivation of such type of algebras is a derivation.

\section{Preliminaries}

In this section we present some known facts
about Lie algebras and their derivations.

Let $L$ be a Lie algebra. For a Lie algebra $L$ consider the following central lower and
derived sequences:
$$
L^1=L,\quad L^{k+1}=[L^k,L^1], \quad k \geq 1,
$$
$$L^{[1]} = 1, \quad L^{[s+1]} = [L^{[s]}, L^{[s]}], \quad s \geq 1.$$

\begin{defn} A Lie algebra $L$ is called
nilpotent (respectively, solvable), if there exists  $p\in\mathbb N$ $(q\in
\mathbb N)$ such that $L^p=0$ (respectively, $L^{[q]}=0$).
\end{defn}

Any Lie algebra $L$ contains a unique maximal solvable (resp. nilpotent) ideal, called the radical (resp. nilradical) of the algebra. A non-trivial Lie algebra is called semi-simple if its radical is zero.

A derivation on a Lie algebra $L$ is a linear map  $d: L \rightarrow L$  which satisfies the Leibniz rule:
$$d([x,y]) = [d(x), y] + [x, d(y)], \quad \text{for any} \quad x,y \in L.$$

The set of all derivations of a Lie algebra $L$ is a Lie algebra with respect to
commutation operation and it is denoted by $Der(L).$

For any element $x\in L$ the operator of
right multiplication $ad_x: L \to L$, defined as $ad_x(z)=[z,x]$ is a derivation, and derivations of this form are called inner derivation. The set of all inner derivations of $L,$ denoted $ad (L),$ is an ideal in $Der(L).$

\begin{defn}
A linear operator $\Delta$ is called a local derivation if for any $x \in L,$ there exists a derivation $d_x: L \rightarrow L$ (depending on $x$) such that
$\Delta(x) = d_x(x).$ The set of all local derivations on $L$ we denote by $LocDer(L).$
\end{defn}

We have following Theorem for the local derivation on semi-simple Lie algebras.
\begin{thm}\cite{Ayupov1} Let $L$ be a finite-dimensional semi-simple Lie algebra. Then any local derivation $\Delta$ on $L$ is a derivation.
\end{thm}

Let $N$ be a finite-dimensional nilpotent Lie algebra. For the matrix of linear operator $ad_x$ denote by $C(x)$ the
descending sequence of its Jordan blocks' dimensions. Consider the
lexicographical order on the set $C(L)=\{C(x)  \ | \ x \in L\}$.

\begin{defn} \label{d4} A sequence
$$\left(\max\limits_{x\in L\setminus L^2} C(x) \right) $$ is said to be the
characteristic sequence of the nilpotent Lie algebra $N.$
\end{defn}

\begin{defn} Nilpotent Lie algebra with characteristic sequence $(n_1, n_2, \dots, n_k, 1)$ is said to be model if there exists a basis $\{e_1, e_2, \dots, e_n\}$ such that
\begin{equation}\begin{cases}[e_i, e_1] = e_{i+1}, & 2 \leq i \leq n_1, \\
[e_{n_1+\dots + n_{j-1}+i}, e_1] = e_{n_1+\dots + n_{j-1}+i+1}, & 2 \leq j \leq k, \ 2 \leq i \leq n_j,\\
 \end{cases}\end{equation}
where omitted products are equal to zero.
\end{defn}

\section{Local derivation of solvable Lie algebras with abelian nilradical}

In this section we investigate solvable Lie algebras with abelian nilradical.
First we consider the following example.
\begin{exam}\label{exam3.1}
Consider following three-dimensional solvable Lie algebras with two-dimensional abelian nilradical.
$$\begin{array}{lllll} L_1:&  [e_2, e_1] = e_2, &  [e_3, e_1] = e_3;\\[1mm]
L_2:&  [e_2, e_1] = e_2+e_3,& [e_3, e_1]  = e_3.\end{array}$$
Any local derivation of $L_1$ is a derivation, but $L_2$ admits a local derivation which is not a derivation.

Indeed, by the direct calculation we obtain that the matrix form of the derivations of algebras  $L_1$ and $L_2,$ respectively have the following forms:
$$Der(L_1)=\left(\begin{array}{ccc}
0&\xi_{1,2}&\xi_{1,3}\\
0&\xi_{2,2}&\xi_{2,3}\\
0&\xi_{3,2}&\xi_{3,3}
\end{array}\right), \quad Der(L_2)=\left(\begin{array}{ccc}
0&\xi_{1,2}&\xi_{1,3}\\
0&\xi_{2,2}&\xi_{2,3}\\
0&0&\xi_{2,2}
\end{array}\right).$$

It is not difficult to show that any local derivation on $L_1$ is a derivation and
linear operator $\Delta$ defined as $\Delta(e_1) = 0, \ \Delta(e_2) = 0, \  \Delta(e_3) = e_3$ on $L_2$ is a local derivation which is not a derivation.
\end{exam}

Let $L$ be a solvable Lie algebra with abelian nilradical $N$ and let $dimN=n,$ $dimL =n+1.$
Take a basis $\{x, e_1, e_2, \dots, e_n\}$ of $L$ such that
$\{e_1, e_2, \dots, e_n\}$ a basis of $N.$ It is known that the operator of right multiplication  $ad_x$ is a non nilpotent operator on $N$ [6].
Moreover, such solvable algebras characterized by the operator $ad_x,$ i.e.,
two solvable algebras with abelian nilradical $N$ and one-dimensional complementary space are isomorphic if and only if the
corresponding operators of right multiplication have the same Jordan forms.


\begin{thm}\label{thm3.2} Let $L$ be a solvable Lie algebra with the abelian nilradical $N$ and the dimension of the complementary space to the nilradical is equal to one.
Any local derivation on $L$ is a derivation if and only if $ad_x$ is a diagonalizable operator.
\end{thm}

\begin{proof}  Let $ad_x$ diagonalized operator. Then  there exists a basis $\{e_1, e_2, \dots, e_n\}$ of $N$ such that the Jordan form of the operator $ad_x$ on this basis has the form:
$$\left(\begin{array}{cccccccc}
\lambda_1&0&0&\dots&0\\
0&\lambda_2&0&\dots&0\\
0&0&\lambda_3&\dots&0\\
\dots&\dots&\dots&\dots&\dots\\
0&0&0&\dots&\lambda_n\\
\end{array}\right)$$

Consequently,
$$[e_i,x]=\lambda_ie_i, \ \ 1\leq i\leq n.$$

Let $d\in Der(L)$, then
$$\left\{\begin{array}{ll}
d(x)=\beta_1e_1+\beta_2e_2+\dots+\beta_ne_n,\\
d(e_i)=\alpha_{i,1}e_1+\alpha_{i,2}e_2+\dots +\alpha_{i,n}e_n, \ 1\leq i\leq n.\\
\end{array}\right.$$

From the property of derivation we have $$d([e_i,x])=[d(e_i),x]+[e_i,d(x)]=$$
$$=[\alpha_{i,1}e_1+\alpha_{i,2}e_2+\dots+\alpha_{i,n}e_n,x]+[e_i, \beta_1e_1+\beta_2e_2+\dots+\beta_ne_n] = $$
$$=\alpha_{i,1}\lambda_1e_1+\alpha_{i,2}\lambda_2e_2+\dots+\alpha_{i,n}\lambda_ne_n.$$

On the other hand,
$$d([e_i,x])=\lambda_id(e_i) = \lambda_i(\alpha_{i,1}e_1+\alpha_{i,2}e_2+ \dots+\alpha_{i,n}e_n).$$

Comparing the coefficients at the basis elements we obtain
\begin{equation}\label{eq3.1} \alpha_{i,j}(\lambda_i - \lambda_j)=0, \quad 1 \leq j \leq n.\end{equation}

\textbf{Case 1.} Let $\lambda_i\not=\lambda_j,$ for any $i,j (i\neq j),$ then we get
$\alpha_{i,j}=0$ for $i\neq j.$ Thus, we have that the matrix of the derivations of $L$ has the form:
$$Der(L)=\left(\begin{array}{cccccccc}
0&\beta_1&\beta_2&\dots&\beta_n\\
0&\alpha_{1,1}&0&\dots&0\\
0&0&\alpha_{2,2}&\dots&0\\
\dots&\dots&\dots&\dots&\dots\\
0&0&0&\dots&\alpha_{n,n}\\
\end{array}\right)$$

\textbf{Case 2.} Let $\lambda_i=\lambda_j$ for some $i$ and $j.$ Without loss of generality, we can assume that 
$$\lambda_1=\dots=\lambda_s, \ \lambda_{s+1}=\dots=\lambda_{s+p}, \quad \dots \quad \lambda_{n-q}=
\dots=\lambda_n.$$

From the equality \eqref{eq3.1} we obtain that
the matrix form of $Der(L)$ is the following:
$$\left(\begin{array}{cccccccccccccccc}
0&\beta_1&...&\beta_s&\beta_{s+1}&\dots&\beta_{s+p}&\dots&\beta_{n-q}&\dots&\beta_n\\
0&\alpha_{1,1}&\dots&\alpha_{1,s}&0&\dots&0&\dots&0&\dots&0\\
\dots&\dots&\dots&\dots&\dots&\dots&\dots&\dots&\dots&\dots&\dots\\
0&\alpha_{s,1}&\dots&\alpha_{s,s}&0&\dots&0&\dots&0&\dots&0\\
0&0&\dots&0&\alpha_{s+1,s+1}&\dots&\alpha_{s+1,s+p}&\dots&0&\dots&0\\
\dots&\dots&\dots&\dots&\dots&\dots&\dots&\dots&\dots&\dots&\dots\\
0&0&\dots&0&\alpha_{s+p,s+1}&\dots&\alpha_{s+p,s+p}&\dots&0&\dots&0\\
\dots&\dots&\dots&\dots&\dots&\dots&\dots&\dots&\dots&\dots&\dots\\
0&0&\dots&0&0&\dots&0&\dots&\alpha_{n-q,n-q}&\dots&\alpha_{n-q,n}\\
\dots&\dots&\dots&\dots&\dots&\dots&\dots&\dots&\dots&\dots&\dots\\
0&0&\dots&0&0&\dots&0&\dots&\alpha_{n,n-q}&\dots&\alpha_{n,n}\\
\end{array}\right)$$

Let $\Delta$ be a local derivation on $L$ and let
$$\left\{\begin{array}{ll}
\Delta(x)=\xi x+\xi_1e_1+\xi_2e_2+\dots+\xi_ne_n,\\
\Delta(e_i)=\zeta_ix+\zeta_{i,1}e_1+\zeta_{i,2}e_2+\dots +\zeta_{i,n}e_n, \ 1\leq i\leq n.\\
\end{array}\right.$$

Considering the equalities $\Delta(x) = d_x(x)$ and $\Delta(e_i) = d_{e_i}(e_i)$ for $1\leq i\leq n,$ we conclude that $\Delta$ is a derivation.  Therefore,  any local derivation on $L$ is a derivation.

Now let Jordan form of the operator $ad_x$ be
$$ad_x=\left(\begin{array}{cccccccccc}
J_1&0&\dots&0\\
0&J_2&\dots&0\\
\dots&\dots&\dots&\dots\\
0&0&\dots&J_s\\
\end{array}\right).$$
and suppose that there exists a Jordan block with order $k (k\geq 2).$
Without loss of generality one can assume that $J_1$ has order $k\geq 2$.
Then the table of multiplication of $L$ has the form:
$$\left\{\begin{array}{ll}
e_ix=\lambda_1e_i+e_{i+1},& 1 \leq i \leq k-1,\\
e_kx=\lambda_1e_k,\\
e_ix=\lambda_ie_i + \mu_ie_{i+1}, & { k+1}\leq i\leq n,
\end{array}\right.$$
where $\mu_i = 0;1.$

By the direct verification of the property of derivation
we obtain that the general form of the matrix of $Der(L)$ is
$$\left(\begin{array}{cccccccccccccc}
0&\beta_1&\beta_2&\dots&\beta_{k-1}&\beta_k&\beta_{k+1}&\dots&\beta_n\\
0&\alpha_{1,1}&\alpha_{1,2}&\dots&\alpha_{1,k-1}&\alpha_{1,k}&0&\dots&0\\
0&0&\alpha_{1,1}&\dots&\alpha_{1,k-2}&\alpha_{1,k-1}&0&\dots&0\\
\dots&\dots&\dots&\dots&\dots&\dots&\dots&\dots\\
0&0&0&\dots&\alpha_{1,1}&\alpha_{1,2}&0&\dots&0\\
0&0&0&\dots&0&\alpha_{1,1}&0&\dots&0\\
0&0&0&\dots&0&0&H_2&\dots&0\\
\dots&\dots&\dots&\dots&\dots&\dots&\dots&\dots&\dots\\
0&0&0&\dots&0&0&0&\dots&H_s\\
\end{array}\right)$$
where $H_i$ are the block matrices with the same dimension of Jordan blocks $J_i$.

Consider the linear operator $\Delta: L \rightarrow L$ defined by
$$\Delta(x) = 0, \quad  \Delta(e_i) = e_i, \quad 1 \leq i  \leq k-1,$$
$$\Delta(e_k) = 2e_k, \quad \Delta(e_i) = 0, \quad k+1 \leq i  \leq n.$$

It is obvious  that $\Delta$ is not a derivation. We show that $\Delta$ is a local derivation.

Indeed,
for any element $y = \gamma x + \eta_1 e_1 + \eta_2 e_2 + \dots + \eta_n e_n \in L$ we consider
$$\Delta(y) = \Delta(\gamma x + \eta_1 e_1 + \eta_2 e_2 + \dots + \eta_n e_n) = \eta_1 e_1 + \eta_2 e_2 + \dots +\eta_{k-1} e_{k-1} +2\eta_k e_k.$$

Consider the derivation $d_y$ such that
$$d_y(x) =0, \qquad d_y(e_i) =0, \quad k+1 \leq i \leq n,$$
$$d_y(e_i) = \alpha_{1,1}e_i + \alpha_{1,2}e_{i+1} + \dots + \alpha_{1,k-i+1}e_k, \quad 1 \leq i \leq k.$$

Then
$$d_y(y) = d_y(\gamma x + \eta_1 e_1 + \eta_2 e_2 + \dots + \eta_n e_n)= $$
$$=\eta_1\alpha_{1,1} e_1 + (\eta_2\alpha_{1,1} + \eta_1\alpha_{1,2})e_2 + \dots +  (\eta_k \alpha_{1,1} + \eta_{k-1}\alpha_{1,2} + \dots + \eta_{1}\alpha_{1,{k}})e_k.$$

Let is show that for an appropriate choice of the parameters $\alpha_{i,j}$ that $\Delta(y) = d_y(y).$ This is satisfied if 
$$\begin{cases} \eta_1 = \eta_1\alpha_{1,1},\\ \eta_2 = \eta_2\alpha_{1,1} + \eta_1\alpha_{1,2}, \\ \dots\dots\dots\dots\dots\dots\dots \\
\eta_{k-1} = \eta_{k-1}\alpha_{1,1} + \eta_{k-2}\alpha_{1,2} + \dots + \eta_{1}\alpha_{1,{k-1}}, \\ 2\eta_k =
\eta_k \alpha_{1,1} + \eta_{k-1}\alpha_{1,2} +  \dots + \eta_{1}\alpha_{1,{k}}.
\end{cases}$$

Note that this system of equations has a solution with respect to $\alpha_{i,j}$ for any parameters $\eta_i.$ Indeed:
\begin{itemize}
\item if $\eta_1 \neq 0,$ then $$\alpha_{1,1} =1, \ \alpha_{1,2} = \dots = \alpha_{1,k-1} = 0, \ \alpha_{1,{k}} = \frac {\eta_k} {\eta_1},$$
\item if $\eta_1=\dots = \eta_{s-1} =0$ è $\eta_s\neq 0, \ 2 \leq s \leq k-1,$  then  $$\alpha_{1,1} =1, \ \alpha_{1,2} = \dots = \alpha_{1,k-s} = 0, \ \alpha_{1,{k-s+1}} = \frac {\eta_k} {\eta_s},$$
\item if  $\eta_1=\dots = \eta_{k-1} =0$ è $\eta_k\neq 0$ then we have $\alpha_{1,1} =2.$
\end{itemize}

Hence, $\Delta$ is a local derivation.

\end{proof}

Now we consider solvable Lie algebras with abelian nilradical and maximal complementary vector space.
It is known that the maximal dimension of complementary space for solvable Lie algebras with $n$-dimensional abelian nilradical is equal to $n.$ Moreover, up to isomorphism there exist only one such solvable Lie algebra with the following non-zero multiplications:
$$L_n: \left[e_i, x_i\right]=e_i, \quad 1 \leq i \leq n.$$

\begin{thm} 
Any local derivation on $L_n$ is a derivation.

\end{thm}

\begin{proof}
First, we describe the derivations of the algebra $L_n.$ Let $d\in Der(L_n)$ then we have $$d(e_i)=\sum\limits_{j=1}^n\alpha_{i,j}e_j, \quad d(x_i)=\sum\limits_{j=1}^n\beta_{i,j}e_j, \quad 1 \leq i \leq n.$$

Using the property of derivation $d([e_i,x_j])=[d(e_i),x_j]+[e_i,d(x_j)]$ we obtain that $\alpha_{i,j}=0$ for  $i \neq j.$
From the equality $d([x_i,x_j])=[d(x_i),x_j]+[x_i,d(x_j)]$ we get $\beta_{i,j}=0$ for  $i \neq j.$
Therefore, we have that any derivation of $L_n$ has the following form:
$$d(e_i)=\alpha_{i}e_i, \quad d(x_i)=\beta_{i}e_i, \quad 1 \leq i \leq n.$$

Let $\Delta$ be a local derivation on $L_n$, then $\Delta(e_i) = d_{e_i}(e_i) = \gamma_{i}e_i$ and $\Delta(x_i) = d_{x_i}(x_i) = \delta_{i}e_i.$ Thus, $\Delta$ is a derivation.
\end{proof}

\section{Local derivation of solvable Lie algebras with model nilradical}

Let $L$ be a solvable Lie algebra and its nilradical is a model algebra $N$.
Let the characteristic sequence of $N$ be equal to $(n_1, n_2, \dots, n_k, 1).$ Then the multplication of $N$ has the following form:
\[\begin{cases}[e_i, e_1] = e_{i+1}, & 2 \leq i \leq n_1, \\
[e_{n_1+\dots + n_{j-1}+i}, e_1] = e_{n_1+\dots + n_{j-1}+i+1}, & 2 \leq j \leq k, \ 2 \leq i \leq n_j.\\
 \end{cases}\]

It is known that the maximal dimension of the complementary space of solvable Lie algebra with model nilpotent Lie algebra with
characteristic sequence $(n_1, n_2, \dots, n_k, 1)$ is equal to $k+1.$  Let $L=Q+N$ be a solvable Lie algebra with $dim Q=k+1.$ Then $L$ has a basis $\{x_1, x_2, \dots, x_{k+1}, e_1, e_2, \dots, e_n\}$ such that
the table of multiplication of $L$ has the form
\[L_{k+1}(N):\left\{\begin{array}{lll}[e_i, e_1] = e_{i+1}, & 2 \leq i \leq n_1, \\[1mm]
[e_{n_1+\dots + n_{j-1}+i}, e_1] = e_{n_1+\dots + n_{j-1}+i+1}, & 2 \leq j \leq k, & 2 \leq i \leq n_j,\\[1mm]
[e_i, x_1] = ie_i, & 1 \leq i \leq n,\\[1mm]
[e_i, x_2] = e_i, & 2 \leq i \leq n_1+1,\\[1mm]
[e_{n_1+\dots + n_{j-2}+i}, x_j] = e_{n_1+\dots + n_{j-2}+i}, & 3 \leq j \leq k+1,  & 2 \leq i \leq n_{j-1}+1.
 \end{array}\right.\]

In \cite{Ancochea} it is proved that any derivation of $L_{k+1}(N)$ is inner for any characteristic sequence $(n_1, n_2, \dots, n_k, 1)$.

In this section we investigate local derivations on $L_{k+1}(N).$
For any local derivation  $\Delta$ on $L_{k+1}(N)$  we have
$$\Delta(x_j) = d_{x_j}(x_j) = [x_j, a_j], \quad 1\leq j \leq k+1,$$
$$\Delta(e_i) = d_{e_i}(e_i) = [e_i, b_i], \quad 1\leq i \leq n.$$

Put
$$a_j = \sum\limits_{p=1}^{k+1}\alpha_{j,p}x_p+ \sum\limits_{p=1}^{n}\beta_{j,p}e_p, \quad b_i = \sum\limits_{p=1}^{k+1}\gamma_{i,p}x_p+\sum\limits_{p=1}^{n}\delta_{i,p}e_p.$$

Then we obtain
$$\Delta(x_1) = -\sum\limits_{p=1}^{n}p\beta_{1,p}e_p, \quad
\Delta(x_j) =  -\sum\limits_{p=n_1+\dots+n_{j-2}+2}^{n_1+\dots+n_{j-1}+1}\beta_{j,p}e_p, \quad 2 \leq j \leq k+1.$$
$$\Delta(e_1) = \gamma_{1,1}e_1 - \sum\limits_{j=2}^{k+1}\sum\limits_{p= n_1+\dots+n_{j-2}+2}^{n_1+\dots+n_{j-1}}\delta_{1,i}e_{p+1},$$
$$\Delta(e_i) = (i\gamma_{i,1} + \gamma_{i,j})e_i + \delta_{i,1} e_{i+1}, \quad n_1+\dots+n_{j-2}+2 \leq i \leq n_1+\dots+n_{j-1}, \ 2 \leq j \leq k-1 $$
$$\Delta(e_i) = (i\gamma_{i,1} + \gamma_{i,j})e_i, \quad  i = n_1+\dots+n_{j-1}+1, \ 2 \leq j \leq k-1.$$

\begin{prop}\label{prop1} There exists $y \in L_{k+1}(N),$ such that $\Delta(x_j) = [x_j, y]$ for $1 \leq j \leq k+1.$

\end{prop}

\begin{proof}

For any $j (2 \leq j \leq k+1)$ and for the fixed $s (n_1 + \dots + n_{j-2}+2\leq s \leq n_1 + \dots + n_{j-1}+1)$ we consider
$$\Delta(s x_j -  x_1) =s\Delta(x_j) - \Delta(x_1) = \sum\limits_{p=1}^{n}p\beta_{1,p}e_p - s \sum\limits_{p=n_1+\dots+n_{j-2}+2}^{n_1+\dots+n_{j-1}+1}\beta_{j,p}e_p.$$

On the other hand $$\Delta(s x_j -  x_1) = [s x_j -  x_1, y_{s x_j -  x_1}] = [s x_j -  x_1, \sum\limits_{p=1}^{k+1}A_{s x_j -  x_1,p}x_p+ \sum\limits_{p=1}^{n}B_{s x_j -  x_1,p}e_p] =$$
$$= \sum\limits_{p=1}^{n}pB_{s x_j -  x_1,p}e_p - s \sum\limits_{p=n_1+\dots+n_{j-2}+2}^{n_1+\dots+n_{j-1}+1} B_{s x_j -  x_1,p}e_p.$$

Comparing coefficients at the basis element of $e_s$ we have
$s (\beta_{1,s} - \beta_{j,s}) =0$ which implies $\beta_{j,s} = \beta_{1,s}.$

Since $j (2 \leq j \leq k+1)$ and $s (n_1 + \dots + n_{j-2}+2\leq s \leq n_1 + \dots + n_{j-1}+1)$ we have that $\beta_{j,p} = \beta_{1,p}$ for any $j$ and $p.$

If we take an element $y=\sum\limits_{p=1}^n\beta_{1,p}e_p,$ then we have
$$\Delta(x_j) = [x_j, y], \quad 1 \leq j \leq k+1.$$

\end{proof}

Therefore without loss of generality we can use $\beta_{p}$ instead of $\beta_{1,p}.$ Thus we have

\begin{equation}\label{eq4.1}\Delta(x_1) = -\sum\limits_{p=1}^{n}p\beta_{p}e_p, \quad
\Delta(x_j) =  -\sum\limits_{p=n_1+\dots+n_{j-2}+2}^{n_1+\dots+n_{j-1}+1}\beta_{p}e_p, \quad 2 \leq j \leq k+1.\end{equation}

Now we consider the value of local derivations on the generators of $N$ (which algebraically generate the basis), i.e.,
$e_1, e_2, e_{n_1+2}, \dots, e_{n_1+\dots + n_{k-1}+2}.$

\begin{prop}\label{prop2} There exists $z \in L_{k+1}(N),$ such that $$\Delta(x_j) = [x_j, z], \ \Delta(e_1) = [e_1, z], \
\Delta(e_{n_1+\dots + n_{j-2}+2}) = [e_{n_1+\dots + n_{j-2}+2}, z], \quad  2 \leq j \leq k+1.$$

\end{prop}

\begin{proof}

By Proposition \ref{prop1} we have that there exist $y\in L_{k+1}(N)$ such that $\Delta(x_j) = [x_j, y].$

Consider $\Delta(x_1 -  3x_2 - e_1)$ Using the equality \eqref{eq4.1} we have
$$ \Delta(x_1 -  3x_2 - e_1) = -\sum\limits_{p=1}^{n}p\beta_{p}e_p + 3\sum\limits_{p=2}^{n_1+1}\beta_{p}e_p - \gamma_{1,1}e_1 + \sum\limits_{j=2}^{k+1}\sum\limits_{p= n_1+\dots+n_{j-2}+2}^{n_1+\dots+n_{j-1}}\delta_{1,i}e_{p+1}$$

On the other hand
$$\Delta(x_1 -  3x_2 - e_1) = [x_1 -  3x_2 - e_1, y_{x_1 -  3x_2 - e_1}] = [x_1 -  3x_2 - e_1, \sum\limits_{p=1}^{k+1}A_{x_1 -  3x_2 - e_1,p}x_p+ \sum\limits_{p=1}^{n}B_{x_1 -  3x_2 - e_1,p}e_p]=$$
$$= - \sum\limits_{p=1}^{n}pB_{x_1 -  3x_2 - e_1,p}e_p + 3\sum\limits_{p=2}^{n_1+1}B_{x_1 -  3x_2 - e_1,p}e_p - A_{x_1 -  3x_2 - e_1,1}x_1 + \sum\limits_{j=2}^{k+1}\sum\limits_{p= n_1+\dots+n_{j-2}+2}^{n_1+\dots+n_{j-1}}B_{x_1 -  3x_2 - e_1,p}e_{p+1} $$

Comparing the coefficients at the basis elements $e_2$ and  $e_3,$ we get
$$B_{x_1 -  3x_2 - e_1,2} = \beta_{2}, \quad B_{x_1 -  3x_2 - e_1,2} = \delta_{1,2}, $$ which implies $\delta_{1,2} = \beta_{2}.$

Considering
$\Delta(x_1 -  (i+1)x_2 - e_1)$ inductively we obtain that $\delta_{1,i} = \beta_{i}$ for $2 \leq i \leq n_1.$  Indeed,
$$ \Delta(x_1 -  (i+1)x_2 - e_1) = -\sum\limits_{p=1}^{n}p\beta_{p}e_p + (i+1)\sum\limits_{p=2}^{n_1+1}\beta_{p}e_p - \gamma_{1,1}e_1 + \sum\limits_{j=2}^{k+1}\sum\limits_{p= n_1+\dots+n_{j-2}+2}^{n_1+\dots+n_{j-1}}\delta_{1,p}e_{p+1}$$

On the other hand
$$\Delta(x_1 -  (i+1)x_2 - e_1) = [x_1 -  (i+1)x_2 - e_1, y_{x_1 -  (i+1)x_2 - e_1}] = $$$$[x_1 -  (i+1)x_2 - e_1, \sum\limits_{p=1}^{k+1}A_{x_1 -  (i+1)x_2 - e_1,p}x_p+ \sum\limits_{p=1}^{n}B_{x_1 -  (i+1)x_2 - e_1,p}e_p]=$$
$$= - \sum\limits_{p=1}^{n}pB_{x_1 -  (i+1)x_2 - e_1,p}e_p + (i+1)\sum\limits_{p=2}^{n_1+1}B_{x_1 -  (i+1)x_2 - e_1,p}e_p - $$$$ -A_{x_1 -  (i+1)x_2 - e_1,1}x_1 + \sum\limits_{j=2}^{k+1}\sum\limits_{p= n_1+\dots+n_{j-2}+2}^{n_1+\dots+n_{j-1}}B_{x_1 -  (i+1)x_2 - e_1,p}e_{p+1} $$

Comparing the coefficients at the basis elements $e_2, e_3, \dots e_{i+1},$ we get
\begin{equation}\label{eq4.2} \begin{cases} (i-1)B_{x_1 -  (i+1)x_2 - e_1,2} = (i-1)\beta_{2}, \\
(i-2)B_{x_1 -  (i+1)x_2 - e_1,3} + B_{x_1 -  (i+1)x_2 - e_1,2} = (i-2)\beta_{3} + \delta_{1,2},\\
(i-3)B_{x_1 -  (i+1)x_2 - e_1,4} + B_{x_1 -  (i+1)x_2 - e_1,3} = (i-3)\beta_{4} + \delta_{1,3},\\
\dots\dots\dots\dots\dots\dots\dots\dots\dots\dots\dots\dots\dots\dots \\
B_{x_1 -  (i+1)x_2 - e_1,i} + B_{x_1 -  (i+1)x_2 - e_1,i-1} = \beta_{i} + \delta_{1,i-1},\\
B_{x_1 -  (i+1)x_2 - e_1,i} = \delta_{1,i}.
\end{cases} \end{equation}

By induction hypotheses we have $\delta_{1,2} = \beta_2, \ \delta_{1,3} = \beta_3, \ \delta_{1,i-1} = \beta_{i-1}$ and from \eqref{eq4.2} we obtain that $\delta_{1,i} = \beta_i.$

In a similar way considering $\Delta(x_1 -  (n_1+\dots + n_{j-2}+i+1)x_{j} - e_1)$ for $3 \leq j \leq k+1,$ $2 \leq i \leq n_{j-1}$ we obtain $$\delta_{1,i} = \beta_1, \quad n_1+\dots + n_{j-2}+2 \leq i \leq n_1+\dots + n_{j-1}, \ 2 \leq j \leq k+1,$$

From $\Delta(x_1 -  (n_1+\dots + n_{j-2}+3)x_{j} - e_{n_1+\dots + n_{j-2}+2})$ for $2 \leq j \leq k+1,$
 we obtain $$\delta_{n_1+\dots + n_{j-2}+2, 1} = \beta_{1}, \quad 2 \leq j \leq k+1.$$

Therefore, we obtain that for the element
$$z=y + \gamma_{1,1}x_1+\sum\limits_{j=2}^{k+1}(\gamma_{n_1+\dots+n_{j-2}+2, 3} + (n_1+\dots+n_{j-2}+2)(\gamma_{n_1+\dots+n_{j-2}+2, 1}-\gamma_{1,1}))x_{j}$$
the following equalities are hold
$$\Delta(x_j) = [x_j, z], \ \Delta(e_1) = [e_1, z], \
\Delta(e_{n_1+\dots + n_{j-2}+2}) = [e_{n_1+\dots + n_{j-2}+2}, z], \quad  2 \leq j \leq k+1.$$

\end{proof}

From Proposition \ref{prop1} and \ref{prop2}, we have that for any local derivation $\Delta$ there exist an element $z\in L_{k+1}(N)$ such that
$\Delta(x) = [z,x]$ for each generator of $x\in L_{k+1}(N).$

Thus, if we put $z = \sum\limits_{p=1}^{k+1}\gamma_{p}e_p + \sum\limits_{p=1}^{n}\beta_{p}e_p,$ then we have
\begin{equation}\label{eq4.3}\Delta(x_1) = -\sum\limits_{p=1}^{n}p\beta_{p}e_p, \quad
\Delta(x_j) =  -\sum\limits_{p=n_1+\dots+n_{j-2}+2}^{n_1+\dots+n_{j-1}+1}\beta_{p}e_p, \quad 2 \leq j \leq k+1.
\end{equation}

$$\Delta(e_1) = \gamma_{1}e_1 - \sum\limits_{j=2}^{k+1}\sum\limits_{p= n_1+\dots+n_{j-2}+2}^{n_1+\dots+n_{j-1}}\beta_{p}e_{p+1},$$
$$\Delta(e_{n_1+\dots + n_{j-2}+2}) = \left((n_1+\dots + n_{j-2}+2)\gamma_1 + \gamma_j\right)e_{n_1+\dots + n_{j-2}+2} + \beta_{1}e_{n_1+\dots + n_{j-2}+3},  \quad  2 \leq j \leq k+1.$$

\begin{thm}
Any local derivation on the solvable Lie algebra $L_{k+1}(N)$ is a derivation.
\end{thm}

\begin{proof} First we prove of the Theorem for case $k=1,$ i.e., the characteristic sequence of the model nilradical $N$ is $(n,1).$
 Then we have the basis $\{x_1, x_2, e_1, e_2, \dots , e_{n+1}\}$ of $L_{2}(N)$ such that
$$\begin{cases}[e_i, e_1] =e_{i+1}, & 2 \leq i \leq n, \\[1mm]
[e_i, x_1] = ie_i, & 1 \leq i \leq n+1,\\[1mm]
[e_i, x_2] = e_i, & 2 \leq i \leq n+1.\end{cases}$$

Let $\Delta$ be a local derivation of $L_{2}(N).$ By Propositions \ref{prop1} and \ref{prop2} we have
$$\begin{cases} \Delta(e_1) = \gamma_{1}e_1 - \sum\limits_{k=2}^{n} \beta_{k}e_{k+1},& \Delta(x_1) =  \sum\limits_{k=1}^{n+1} k\beta_{k}e_k, \\
  \Delta(e_2) = (2\gamma_1+\gamma_{2})e_2 + \beta_{1}e_3, & \Delta(x_2) = - \sum\limits_{k=2}^{n+1} \beta_{k}e_k,\\
\Delta(e_i) = (i\gamma_{i,1} +\gamma_{i,2})e_i + \delta_{i,1}e_{i+1}, & 3 \leq i \leq n,\\
\Delta(e_{n+1}) = ((n+1)\gamma_{n+1,1} +\gamma_{n+1,2})e_{n+1} .\end{cases}$$

Consider $$\Delta(x_1 - (j+1)x_2 + e_j) = - \beta_{1}e_1 + (j\gamma_{j,1} +\gamma_{j,2})e_j + \delta_{j,1}e_{j+1} + \sum\limits_{k=2}^n (j+1-k)\beta_{k}e_k.$$

On the other hand $$\Delta(x_1 - (j+1)x_2 + e_j) = [x_1 - (j+1)x_2 + e_j, A_{j,1}x_1 + A_{j,2}x_2 + \sum\limits_{k=1}^{n+1}B_{j,k}e_k]=$$
$$= - B_{j,1} e_1 +(jA_{j,1} + A_{j,2} + B_{j,j-1})e_j + B_{j,1}e_{j+1}+\sum\limits_{k=2}^{n+1} (j+1-k)B_{j,k}e_k.$$

Comparing the coefficients at the basis elements $e_1$ and $e_{j+1}$ we have  $B_{j,1} = \delta_{j,1}$ and $B_{j,1} = \beta_{1},$ which implies

$$\delta_{j,1} = \beta_{1}, \quad  3 \leq j \leq n.$$

Now consider $$\Delta(x_1 - (j+1)x_2 + e_1 - (j-1)e_2 + \frac 1 {(j-2)!} e_{j+1}) = (\gamma_{1} - \beta_{1}) e_1 - (j-1) (2\gamma_1+\gamma_{2} - \beta_{2}) e_2 - $$
$$-\big((j-1)\beta_1 + \beta_2 - (j-2)\beta_3\big) e_3  -  \big(\beta_{j}-  \frac 1 {(j-2)!}((j+1)\gamma_{j+1,1} + \gamma_{i+1,2})\big) e_{j+1}$$
$$- \sum\limits_{k=4,k \neq j+1}^{n+1}(\beta_{k-1} - (j+1-k)\beta_{k} )e_k.$$

On the other hand $$\Delta(x_1 - (j+1)x_2 + e_1 - (j-1)e_2 + \frac 1 {(j-2)!} e_{j+1}) = $$$$[x_1 - (j+1)x_2 + e_1 - (j-1)e_2 + \frac 1 {(j-2)!} e_{j+1}, A_{j,1}x_1 + A_{j,2}x_2 + \sum\limits_{k=1}^{n+1} B_{j,k}e_k]= $$
$$=(A_{j,1}- B_{j,1}) e_1 - (j-1)(2A_{j,1}+A_{j,2}-B_{j,2})e_2 -\big((j-1)B_{j,1} + B_{j,2} - (j-2)B_{j,3})e_3-$$
$$
\big(B_{j,k}-  \frac 1 {(j-2)!}((j+1)A_{j,1} + A_{j,2})\big) e_{j+1} - \sum\limits_{k=4, k\neq j+1}^{n+1}(B_{j,k-1} - (j+1-k)B_{j,k})e_k.$$

Comparing the coefficients at the basis elements of $e_1, e_2, \dots e_{j+1},$ we have
$$\begin{cases}
A_{j,1}- B_{j,1} = \gamma_{1} - \beta_{1}, \\
2A_{j,2} + A_{j,2}-B_{j,2} = 2\gamma_{1} + \gamma_{2} - \beta_{2}, \\
(j-1)B_{j,1} + B_{j,2} - (j-2)B_{j,3} = (j-1)\beta_1 + \beta_2 - (j-2)\beta_3, \\
B_{j,k-1} - (j+1-k)B_{j,k} = \beta_{k-1} - (j+1-k)\beta_{k},& 4 \leq k \leq j,\\
B_{j,k}-  \frac 1 {(j-2)!}\big((j+1)A_{j,1} + A_{j,2}\big) = \beta_{k}-  \frac 1 {(j-2)!}\big((j+1)\gamma_{j+1,1} + \gamma_{i+1,2}\big)
\end{cases}
$$

From this system of equations considering $(j-1) [1] + [2] + [3] + \sum\limits_{k=4}^{j+1} \frac{(j-2)!} {(j+1-k)!}[K]$ we have that
$$(j+1)\gamma_{j+1,1} + \gamma_{i+1,2} = (j+1)\gamma_{1} + \gamma_{2},$$
where $[K]$ is the $k$-th equation of the previous system.

Therefore, we obtain that  $\Delta(y) = [y, \gamma_{1}x_1 + \gamma_{2}x_2 + \sum\limits_{k=1}^{n+1}\beta_{k}e_k]$ for any $y \in L.$ Hence, $\Delta$ is a derivation.

Now we are able to prove the Theorem for the general case. Let $\Delta$ be a local derivation on $L_{k+1}(N),$ then by Propositions \ref{prop1} and \ref{prop2} we have

\begin{equation}\label{eq4.3}\Delta(x_1) = -\sum\limits_{p=1}^{n}p\beta_{p}e_p, \quad
\Delta(x_j) =  -\sum\limits_{p=n_1+\dots+n_{j-2}+2}^{n_1+\dots+n_{j-1}+1}\beta_{p}e_p, \quad 2 \leq j \leq k+1.
\end{equation}

$$\Delta(e_1) = \gamma_{1}e_1 - \sum\limits_{j=2}^{k+1}\sum\limits_{p= n_1+\dots+n_{j-2}+2}^{n_1+\dots+n_{j-1}}\beta_{p}e_{p+1},$$
$$\Delta(e_{n_1+\dots + n_{j-2}+2}) = \left((n_1+\dots + n_{j-2}+2)\gamma_1 + \gamma_j\right)e_{n_1+\dots + n_{j-2}+2} + \beta_{1}e_{n_1+\dots + n_{j-2}+3},  \quad  2 \leq j \leq k+1.$$

$$\Delta(e_i) = (i\gamma_{i,1} + \gamma_{i,j})e_i + \delta_{i,1} e_{i+1}, \quad n_1+\dots+n_{j-2}+3 \leq i \leq n_1+\dots+n_{j-1}, \ 2 \leq j \leq k-1 $$
$$\Delta(e_i) = (i\gamma_{i,1} + \gamma_{i,j})e_i, \quad  i = n_1+\dots+n_{j-1}+1, \ 2 \leq j \leq k-1.$$

To prove of the Theorem we have to show \begin{equation}\label{eq4.4}\delta_{i,1} = \beta_{i}, \quad n_1+\dots+n_{j-2}+3 \leq i \leq n_1+\dots+n_{j-1}, \ 2 \leq j \leq k-1,\end{equation}
and
\begin{equation}\label{eq4.5}i\gamma_{i,1} + \gamma_{i,j} = i\gamma_{1} + \gamma_{j}, \quad n_1+\dots+n_{j-2}+3 \leq i \leq n_1+\dots+n_{j-1}, \ 2 \leq j \leq k.\end{equation}

Similarly to the case $k=1,$ considering $\Delta(x_1 - (n_1+\dots + n_{j-2}+s+1)x_j + e_{n_1+\dots + n_{j-2}+s})$ we obtain the equality \eqref{eq4.4} and analyzing $$\Delta\left(x_1 - (n_1+\dots + n_{j-2}+s+1)x_j + e_1 - (s-1)e_{n_1+\dots + n_{j-2}+s} + \frac{1}{(s-2)!}e_{n_1+\dots + n_{j-2}+s+1}\right)$$ for $2\leq s \leq n_j-1$ we get the equality \eqref{eq4.5}.

\end{proof}

\subsection{Non model nilradical case}

In this subsection we give some examples of local derivation of solvable Lie algebras with maximal dimension of complementary vector space. Such solvable algebras we call maximal solvable Lie algebras with nilradical $N,$ i.e., we a solvable Lie algebra $L$ with nilradical $N$ is said to be maximal, if there is no algebra $M$ with nilradical $N$ such that $dim(M) > dim(L).$

In the first example we consider solvable Lie algebra with non model nilradical and dimension of complementary vector space equal to the number of generators of the nilradical.

\begin{exam}
Let $N$ be $8$-dimensional nilpotent algebra with multiplication $$\begin{array}{lll}[e_2, e_1] =e_{4}, & [e_4, e_1] =e_{5}, & [e_5, e_1] =e_{6},\\[1mm]
[e_3, e_2] =e_{7}, & [e_7, e_1] =e_{8}, & [e_4, e_3] =-e_{8}.\end{array}$$
It is obvious that this is a non model nilpotent algebra with the characteristic sequence $(4,3,1)$ and the number of generators is equal to $3.$
Moreover, the maximal solvable Lie algebra with nilradical $N$ is $11$-dimensional and has the multiplication:
$$L: \left\{\begin{array}{llllll}[e_2, e_1] =e_{4}, & [e_4, e_1] =e_{5}, & [e_5, e_1] =e_{6},&
[e_3, e_2] =e_{7}, & [e_7, e_1] =e_{8}, & [e_4, e_3] =-e_{8},\\[1mm]
[e_1, x_1] = e_1, & [e_4, x_1] = e_4, & [e_5, x_1] = 2e_5, & [e_6, x_1] = 3e_6, & [e_8, x_1] = e_8,\\[1mm]
[e_2, x_2] = e_2, & [e_4, x_2] = e_4, & [e_5, x_2] = e_5, & [e_6, x_2] = e_6, & [e_7, x_2] = e_7, & [e_8, x_2] = e_8,\\[1mm]
[e_3, x_3] = e_3, &  [e_7, x_3] = e_7, & [e_8, x_3] = e_8.\end{array}\right.$$
Any local derivation on $L$ is a derivation moreover it is an inner derivation, i.e., $$LocDer(L) = Der(L) = ad(L).$$
\end{exam}

Now we consider an example with the dimension of complementary vector space of solvable Lie algebra less than number of  generators of the nilradical.

\begin{exam}
Let $L$ be a maximal solvable Lie algebra with five-dimensional Heisenberg nilradical
$$[e_2, e_1] =e_{5}, \quad [e_4, e_3] =e_{5}.$$

Then $dim(L) = 8$ and $L$ has the multiplication:
$$L: \left\{\begin{array}{lllll} [e_2, e_1] =e_{5}, & [e_4, e_3] =e_{5}, \\[1mm]
[e_1, x_1] = e_1, & [e_2, x_1] = e_2, & [e_3, x_1] = e_3, & [e_4, x_1] = e_4, & [e_5, x_1] = 2e_5,\\[1mm]
[e_1, x_2] = e_2, & [e_2, x_2] = e_2, & [e_3, x_2] = 2e_3,  & [e_5, x_2] = 2e_5,\\[1mm]
[e_1, x_3] = 2e_1, &  [e_3, x_3] = e_3, & [e_4, x_3] = e_4, & [e_5, x_3] = 2e_5.\end{array}\right.$$

Any local derivation of $L$ is a derivation (moreover, inner derivation).
\end{exam}

Now we formulate the following conjecture

\begin{conj} Let $L$ be a maximal solvable Lie algebra with nilradical $N.$ Then any local derivation on $L$  is a derivation.

\end{conj}

\end{document}